\documentclass{amsart}
\usepackage{amsmath}
\usepackage{graphicx}
\usepackage{amssymb}
\usepackage{epstopdf}
\usepackage{tikz-cd}
\usepackage{xcolor}
\usepackage{hyperref}
\usepackage{tikz}
\usepackage{tikz-cd} 
\usetikzlibrary{cd}
\usepackage{framed}
\usepackage[T1]{fontenc}
\usepackage{bm}

\newtheorem{thm}{Theorem}[section]
\newtheorem*{thm*}{Main Theorem}
\newtheorem{lem}[thm]{Lemma}
\newtheorem{prop}[thm]{Proposition}
\newtheorem{defn}[thm]{Definition}

\newtheorem*{conjecture*}{Question JLS}
\newtheorem*{subquestion*}{Subquestion JLS}
\newtheorem*{observation*}{Observation}
\newtheorem*{question*}{Question JLS$^\prime$}

\newtheorem{remark}[thm]{Remark}

\newcommand{\FZ}{\mathbb{Z}}  % Integer ring Z
\newcommand{\FC}{\mathbb{C}}  % Complex field
\newcommand{\FR}{\mathbb{R}}  % Real field
\newcommand{\FQ}{\mathbb{Q}}  % Rational field

\newcommand{\Pmod}[1]{\left({\rm mod\,} #1\right)}
\newcommand{\chif}{\chi_{\rm f}}

\newcommand{\fa}{\mathfrak{a}}

\newcommand{\fm}{\mathfrak{m}}

\newcommand{\lc}{\Big{(}}
\newcommand{\rc}{\Big{)}}

\newcommand{\Gal}{\operatorname{Gal}}

\newcommand{\be}{\mathbf{e}}

\DeclareMathOperator{\Tr}{Tr}

\title[Hecke {\it L}-values with cyclotomic twists]{Cyclotomic fields are generated by cyclotomic Hecke {\it L}-values of totally real fields, II}
\author{Jaesung kwon}
\author{Hae-Sang Sun}
\email{jaesungkwon@snu.ac.kr}
\email{haesang@unist.ac.kr}
                         
\date{\today}

\begin{document}

\begin{abstract}
Jun-Lee-Sun \cite{JLS} posed the question of whether the cyclotomic Hecke field can be generated by a single critical $L$-value of a cyclotomic Hecke character over a totally real field. They provided an answer to this question in the case where the tame Hecke character is trivial. In this paper, we extend their work to address the case of non-trivial Hecke characters over solvable totally real number fields. Our approach builds upon the primary estimation obtained by Jun-Lee-Sun \cite{JLS}, supplemented with new inputs, including global class field theory, duality principles, the analytic behavior of partial Hecke $L$-functions, and the non-vanishing of twisted Gauss sums and Hyper Kloosterman sums.
%Jun-Lee-Sun \cite{JLS} proposed a conjecture that the cyclotomic Hecke field can be generated by a single algebraic $L$-value of a cyclotomic Hecke character over a totally real field, and proved this conjecture when the tame Hecke character is trivial.
%In the present paper, we make a progress on this conjecture for the case of non-trivial Hecke character over solvable totally real number fields by utilizing the main estimation obtained by Jun-Lee-Sun \cite{JLS} with some new inputs, which are global class field theory, duality and analytic behaviour of partial Hecke $L$-function, and non-vanishing of twisted Gauss sums and Hyper Kloosterman sums. %More precisely, let $K$ be a solvable totally real number field, then the field generated by a single cyclotomic Hecke $L$-value over $K$ generates the field contains the values of cyclotomic characters and one of tame characters at the rational primes.
\end{abstract}

\subjclass[2020]{11R42 11R80 11R23}
\keywords{totally real field, cyclotomic character, Hecke {\it L}-function, Hecke field}
\maketitle
\setcounter{tocdepth}{1}
\tableofcontents

\section{Introduction}
Various aspects of special $L$-values of automorphic forms on GL$(n)$ with cyclotomic $\mathbb{Z}_p$-twist have been widely studied. One notable example is the {\it Hecke field generation problem}.
For instance, Luo-Ramakrishnan \cite{luo1997determination} proved that the cyclotomic Hecke field of modular forms is generated by the $p$-power roots of unity and critical modular $L$-values with cyclotomic character twists.
Sun \cite{sun} demonstrated that the cyclotomic Hecke field of modular forms is generated by a single critical modular $L$-value twisted by a cyclotomic character, applying this result to the arithmetic of $L$-values for Hida families. 

Instead of modular forms, one can consider GL(1) automorphic forms over a totally real field $K$, specifically ray class characters over $K$. Let $p$ be a rational prime unramified in $K$. 
A ray class character $\chi$ over $K$ is said to be a {\it cyclotomic character} over $K$ modulo a $p$-power if 
$$
\chi=\psi\circ N_{K/\mathbb{Q}}
$$ 
for some primitive Dirichlet character $\psi$ modulo a $p$-power. Note that such characters form the dual of the Galois group of the {\it cyclotomic $\mathbb{Z}_p$-extension} of $K$, and under the Leopoldt's conjecture, the $\mathbb{Z}_p$-extension of $K$ is unique as $K$ is totally real.

Let $\eta$ be a totally odd primitive ray class character over $K$ modulo $
\mathfrak{m}$. Let us set
$
F:=\mathbb{Q}(\mu_{p^{n_0}(p-1)}),
$ 
where $\mu_M\subset\mathbb{C}$ is the set of $M$-th roots of unity, and $n_0>1$ is the integer such that
$$
\mathbb{Q}(\eta)\cap\mathbb{Q}(\mu_{p^2})=\mathbb{Q}(\mu_{p^{n_0}}).
$$
Let $L_K(0,\xi)$ be a critical $L$-value of a ray class character $\xi$. 
Then, Jun-Lee-Sun \cite{JLS} suggested a question motivated by the results on modular forms, such as those by Luo-Ramakrishnan \cite{luo1997determination} and Sun \cite{sun}, as follows:
\begin{conjecture*} For almost all totally even cyclotomic characters $\chi$ over $K$ modulo a $p$-power, does the following equality hold?
$$
F\big(L_K(0,\eta\chi)\big)=F(\eta,\chi).
$$ 
\end{conjecture*}
Jun-Lee-Sun \cite{JLS} provided an answer to this question when the tame character is trivial.
To solve this problem, they first show that 
$$
\text{``average'' of }L_K(0,\chi)=1+o(1)
$$
as the exponent of the conductor of $\chi$ goes to infinity.
The main ingredients are given as follows: 
\begin{itemize}
\item[(1)] non-singular cone decomposition of the fundamental domain of the real Minkowski space of $K$ quotiented by the action of the group of global units of $K$,
\item[(2)] the square root cancellation on the exponential sum on the unit group.
\end{itemize} 
This method can also be applied to the average of critical $L$-values for non-trivial $\eta$ with minor modifications:
\begin{thm}[Theorem \ref{theorem2}]\label{intro:theorem2}
Let $\chi$ be a totally even primitive cyclotomic character over $K$ modulo a $p$-power. For a positive integer $m$, an estimation on an ``$m$-average'' of $L_K(0,\eta\chi)$ is given by
$$
\text{``$m$-average'' of }L_K(0,\eta\chi)=\sum_{N_{K/\mathbb{Q}}(\mathfrak{a})=m}\eta(\mathfrak{a})+o(1)
$$
as the exponent of $\chi$ goes to infinity, where $\mathfrak{a}$ runs through the integral ideals of $K$ with given norm $m$.
\end{thm}
Let us define the {\it norm average function} $c_\eta:\mathbb{Z}_{>0}\to\mathbb{Z}[\eta]$ attached to $\eta$ as follows:
$$
c_\eta(m):=\sum_{N_{K/\mathbb{Q}}(\mathfrak{a})=m}\eta(\mathfrak{a}).
$$
By applying the estimation provided in Theorem \ref{intro:theorem2}, we derived the following result:
\begin{prop}[Proposition \ref{hecke:field:generation} and Proposition \ref{cyclotomic:generation}]\label{intro:hecke:field:generation}
\begin{itemize}
\item[(1)] For almost all totally even primitive cyclotomic characters $\chi$
 over $K$ modulo a $p$-power,
$$
F(L_K(0,\eta\chi),\chi)=F(c_\eta^\prime,\chi).
$$
\item[(2)] If $c_\eta(m)\neq 0$ for some $m>0$ with $m^{p-1}\not\equiv 1\ (p^2)$, then $F(\chi)\subset F\big(L_K(0,\eta\chi)\big)$.
\end{itemize}
\end{prop}
Therefore, to answer Question JLS, we have to give an answer of the following subquestion: 
\begin{subquestion*}
Does $c_\eta$ determine $\eta$?
\end{subquestion*}
Unfortunately, the norm average $c_\eta$ is difficult to analyze, as $\eta$ cannot generally be factored through the norm map $N_{K/\mathbb{Q}}$.
Hence, we need some new inputs, which will be described in the following subsections.

\subsection{Partial Hecke $L$-function}
Instead of computing the norm average directly, we take an alternative approach. %Note that the Hecke $L$-function of $\eta$ can be written as
%$$
%L_K(s,\eta)=\sum_{m>0}\frac{c_\eta(m)}{n^s}.
%$$
Let us define a {\it partial Hecke $L$-function} as follows:
$$
L_K^\circ(s,\eta):=\sum_{m^{p-1}\not\equiv 1\ (p^2)}\sum_{N_{K/\mathbb{Q}}(\mathfrak{a})=m}\frac{\eta(\mathfrak{a})}{N(\mathfrak{a})^s}=\sum_{m^{p-1}\not\equiv 1\ (p^2)}\frac{c_\eta(m)}{m^s},
$$
where $m$ runs through the positive integers such that $m^{p-1}\not\equiv 1\ (p^2)$.
Let us denote by $c_\eta^\circ$ the restriction of $c_\eta$ to the set of positive integers $m>0$ satisfying $m^{p-1}\not\equiv 1\ (p^2)$.
Let us present our critical observation:
\begin{observation*}
$c_\eta^\circ$ determines at least ``the restriction'' of $\eta_{\rm f}$ 
%to the set of rational primes (except some finite primes) 
if the ``twisted average of Gauss sums'' is non-vanishing. 
\end{observation*}
Let us describe this in detail. Let $\eta_1$ and $\eta_2$ be ray class characters over $K$ modulo $\mathfrak{m}$.
Note that $c_{\eta_1}^\circ=c_{\eta_2}^\circ$ if and only if
\begin{equation}\label{intro:partial:hecke:equation}
L_K^\circ\big(s,\eta_1\cdot(\omega^j\circ N_{K/\mathbb{Q}})\big)=L_K^\circ\big(s,\eta_2\cdot(\omega^j\circ N_{K/\mathbb{Q}})\big)\text{ for any }j.
\end{equation}
Due to the significant properties of partial Hecke $L$-functions, namely 
\begin{itemize}
\item[(a)]
Duality:
$$
L_K^\circ(s,*)\longleftrightarrow L_K^\circ(1-s,\overline{*}),
$$
\item[(b)]
Analytic behavior of $L_K^\circ(s,*)$,
\end{itemize}
we can extract valuable information from equation (\ref{intro:partial:hecke:equation}) that
$$
\big(\eta_{1,\mathrm{f}}(q)-\eta_{2,\mathrm{f}}(q)\big)\sum_{1\neq\psi\in\mathrm{Hom}((\mathbb{Z}/p^2\mathbb{Z})^\times,\mu_p)}\tau(\omega^j\psi\circ N_{K/\mathbb{Q}})\omega^j\psi\big(N_{K/\mathbb{Q}}(\mathfrak{m})\big)=0,
$$
for almost all $q$ and any $j$ (see the proof of Proposition \ref{nonvanishing:averagesum:characters}). 
If we can show that the twisted average of Gauss sums 
\begin{equation}\label{intro:gausssum:average}
\sum_{1\neq\psi\in\mathrm{Hom}((\mathbb{Z}/p^2\mathbb{Z})^\times,\mu_p)}\tau(\omega^j\psi\circ N_{K/\mathbb{Q}})\psi\big(N_{K/\mathbb{Q}}(\mathfrak{m})\big)
\end{equation}
is non-vanishing for some $j$, then $\eta_{1,\mathrm{f}}(q)=\eta_{2,\mathrm{f}}(q)$ for almost all rational primes $q$. 

In conclusion, the non-vanishing of the twisted average of Gauss sums provides a partial answer to Subquestion JLS.

\subsection{Gauss sums and global class field theory}
The average (\ref{intro:gausssum:average}) can be rewritten as follows:
$$
\sum_{\psi\neq 1}\tau(\omega^j\psi\circ N_{K/\mathbb{Q}})\psi\big(N_{K/\mathbb{Q}}(\mathfrak{m})\big)=p\sum_{\substack{\alpha\in (O_K/p^2O_K)^\times \\ N_{K/\mathbb{Q}}(\alpha\mathfrak{m})d_K\equiv 1\ (p^2) }}\omega^j( N_{K/\mathbb{Q}}(\alpha)d_K)\mathbf{e}\lc\frac{\mathrm{Tr}_{K/\mathbb{Q}}(\alpha)}{p^2}\rc,
$$
where $d_K$ is the discriminant of $K$ and $\alpha$ runs over mod $p^2O_K$ with $N_{K/\mathbb{Q}}(\alpha\mathfrak{m})d_K\equiv 1\ (p^2)$. This sum is very difficult to study due to the non-trivial relationship between the trace $\mathrm{Tr}_{K/\mathbb{Q}}$ and the norm $N_{K/\mathbb{Q}}$, except when $K=\mathbb{Q}$. 

At this moment, we impose a new assumption: Assume that $K$ is solvable, i.e., $K$ is a Galois number field that can be exhausted by a sequence of abelian extensions.
Then, for a non-trivial character $\phi\in\widehat{(\mathbb{Z}/p^2\mathbb{Z})^\times}$, we obtain that the Gauss sum $\tau(\phi\circ N_{K/\mathbb{Q}})$ can be expressed as the product of Gauss sums $\tau(\phi)$ of Dirichlet characters $\phi$ by using the duality and the decomposition of $L_K(s,\phi\circ N_{K/\mathbb{Q}})$, which is a part of global class field theory. More precisely,
\begin{equation}\label{intro:gausssum:decomposition}
\tau(\phi\circ N_{K/\mathbb{Q}})=C\phi(d_K)\tau(\phi)^{[K:\mathbb{Q}]}
\end{equation}
for some non-zero constant $C$ depends only on $K$ and the exponent of the conductor of $\phi$, where $d_K$ is the discriminant of $K$ (see Proposition \ref{gausssum:decomposition:2}). Note that a power of the Gauss sum of a Dirichlet character is easier to understand than the Gauss sum of a ray class character. Let us explain this in detail at the next subsection.

\subsection{Higher power of Gauss sums and Hyper Kloosterman sums}
It is well-known that high-powers of the Gauss sum of a Dirichlet character are related to {\it hyper Kloosterman sums}. 
For $M>0$ and $r\in (\mathbb{Z}/M\mathbb{Z})^\times$, let us define hyper Kloosterman sums of dimension $d>0$ and of modulus $M$ as follows: 
$$
\mathrm{Kl}_d(r;M):=\sum_{\substack{n_1,\cdots,n_d\in  (\mathbb{Z}/M\mathbb{Z})^\times \\ n_1\times\cdots\times n_d\equiv r(M) }}\mathbf{e}\Big(\frac{n_1+\cdots+n_d}{M}\Big).
$$
Gurak \cite{Gurak} computed the explicit values of $\mathrm{Kl}_d(r;p^n)$ when $n>1$. 
Especially, for $n>1$, the Kloosterman sum $\mathrm{Kl}_d(r;p^n)$ is non-vanishing if $r\equiv 1 \ (p)$.
If we assume that $K$ is solvable, then the $\omega^j$-weighted average of the sum $(\ref{intro:gausssum:average})$ is given by $\mathrm{Kl}_{d}(r;p^2)$ for some $r\equiv 1\ (p)$, where $d=[K:\mathbb{Q}]$. This implies that the average $(\ref{intro:gausssum:average})$ is non-vanishing for some $j$ (see Proposition \ref{twisted:average:gausssum}). 

\subsection{Main results}
Based on the preceding discussions, we obtain the following results toward Subquestion JLS:
\begin{prop}[Proposition \ref{nonvanishing:averagesum:characters}]
Let $A$ be a finite set of totally odd primitive ray class characters over $K$ modulo $\mathfrak{m}$.
Let $a_\eta$ be constants satisfying
$$
\sum_{\eta\in A}a_\eta c_\eta(m)=0
$$
for any positive integers $m$ such that $m^{p-1}\not\equiv 1\ (p^2)$.
If $K$ is solvable and $p$ does not divide $N_{K/\mathbb{Q}}(\mathfrak{m})$, then we have 
$$
\sum_{\eta\in A} a_\eta \tau(\eta)\eta_{\rm f}(p^2)=0\text{ and }\sum_{\eta\in A} a_\eta \tau(\eta)\eta_{\rm f}(p^2q)=0
$$
for any odd rational primes $q$ does not divide $pN_{K/\mathbb{Q}}(\mathfrak{m})$.
\end{prop}

From the above determination result and the preceding discussion, finally we obtain the following result toward Question JLS:
\begin{thm}[Theorem \ref{main:result}]
Let us assume that $K$ is solvable and $p$ does not divide $N_{K/\mathbb{Q}}(\mathfrak{m})$.
For almost all totally even primitive cyclotomic characters $\chi$ over $K$ modulo a $p$-power,
$$
F\big(L_K(0,\eta\chi)\big)=F(c_\eta^\prime,\chi).
$$
\end{thm}

\begin{thm}[Theorem \ref{main:theorem:final}]\label{intro:main:result}
Let us assume that $K$ is solvable and $p$ does not divide $N_{K/\mathbb{Q}}(\mathfrak{m})$.
For almost all totally even primitive cyclotomic characters $\chi$ over $K$ modulo $p$-power,
$$
F(\eta_{\rm f}^\prime,\chi)\subset F(c_\eta^\prime,\chi)\subset  F\big(L_K(0,\eta\chi)\big)\subset F(\eta,\chi),
$$
where $\eta_{\rm f}^\prime$ is the restriction of $\eta_{\rm f}$ to the set of odd rational primes coprime to $pN_{K/\mathbb{Q}}(\mathfrak{m})$.
\end{thm}

\begin{remark}
We have mentioned that {\rm Jun-Lee-Sun} \cite{JLS} prove the Hecke field generation when $\eta$ is trivial. However, the $L$-value $L_K(0,\eta\chi)$ can be decomposed into the product of Dirichlet $L$-values when $\eta$ is trivial and $K$ is abelian. It is relatively straightforward to prove the Hecke field generation problem for Dirichlet $L$-values since the norm average is trivial in this case. Consequently, one might assume that the methodology of {\rm Jun-Lee-Sun} \cite{JLS} cannot be applied to the case of non-trivial $\eta$ and abelian $K$. Nevertheless, in this paper, we utilize the estimation {\rm Theorem \ref{intro:theorem2} } based on {\rm Jun-Lee-Sun} \cite{JLS}, along with the aforementioned new inputs, to prove the result toward {\rm Question JLS} even when $\eta$ is non-trivial. Thus, this paper, together with the work of {\rm Jun-Lee-Sun} \cite{JLS} plays a significant role in the Hecke field generation problem for {\rm GL(1)} automorphic forms over totally real fields.
\end{remark}

\begin{remark}
We have mentioned in the previous remark that the $L$-function of a cyclotomic character over $K$ can be decomposed into Dirichlet $L$-functions for an abelian extension $K$, but this argument does not hold for a solvable field $K$ in general, since the characters of each layer of abelian extensions cannot be factored through by the norm map in general. But our trick (see Proposition \ref{gausssum:decomposition:2}) allows a decomposition of Gauss sums of cyclotomic characters even $K$ is solvable (so that non-abelian), which is crucial for the proof of the determination arguments.
\end{remark}

\subsection{Selberg orthonormality and Langlands functoriality}
One possible approach to giving the complete answer to Question JLS is through the Joint universality of Hecke $L$-functions.
Let $\eta_1$ and $\eta_2$ be tame ray class characters over $K$. Then, 
$c_{\eta_1}^\prime=c_{\eta_2}^\prime$
if and only if
\begin{equation}\label{intro:equation:universality}
L_K(s,\eta_1\chi)=L_K(s,\eta_2\chi)
\end{equation}
for any cyclotomic characters $\chi$ over $K$ modulo a $p$-power.
If the joint universality holds for Hecke $L$-functions, then equation (\ref{intro:equation:universality}) implies that $\eta_1=\eta_2$, thereby legitimating Question JLS.
It is well known that if the following quantities
\begin{equation}\label{intro:selberg:ortho}
\sum_{q\leq x}c_{\eta_i}(q)\overline{c_{\eta_j}(q)}
\end{equation}
for $i,j=1,2$ satisfying some ``good'' estimation on $x$, then $L_K(s,\eta_1)$ and $L_K(s,\eta_2)$ are jointly universal (see Lee-Nakamura-Pa{\'n}kowski \cite{LNP}).

Let $d=[K:\mathbb{Q}]$. If $K$ is solvable, then by the Langlands functoriality for solvable induction proved by Arthur-Clozel \cite[Section 6]{basechange}, the induced representation $\pi_\eta:=\mathrm{Ind}_{GL_1(\mathbb{A}_K)}^{GL_d(\mathbb{A}_\mathbb{Q})}(\eta)$ is automorphic, whose Satake parameter is given by
$$
(\eta\circ g)_{g\in\mathrm{Gal}(K/\mathbb{Q})},
$$
where $\mathrm{Gal}(K/\mathbb{Q})$ acts on the ray class groups.
From this, some ``twisted sum'' of the Rankin-Selberg $L$-function $L(s,\pi_{\eta_1}\times\pi_{\overline{\eta_2}})$ can be written as follows:
$$
\text{``twisted sum'' of }L(s,\pi_{\eta_i}\times\pi_{\overline{\eta_j}})=\sum_{m>0}\frac{c_{\eta_i}(m)\overline{c_{\eta_j}(m)}}{m^s}.
$$
If we take the logarithmic derivative of the above equation and estimate it using the analytic properties of $L(s,\pi_{\eta_i}\times\pi_{\overline{\eta_j}})$, as is done for the Riemann $\zeta$-function to prove the prime number theorem, we may obtain an estimation for equation (\ref{intro:selberg:ortho}). 
It is one of the author's future projects that proving the generation problem by adopting this method.
Let us conclude this section with the following important remarks:

\begin{remark}
Our discussion in this subsection provides an evidence that justifies our assumption in the main result {\rm (Theorem \ref{intro:main:result})} that $K$ is solvable, as there has been no progress on the Langlands functoriality on the induction for $\mathrm{GL}_d(\mathbb{A}_\mathbb{Q})\rightarrow\mathrm{GL}_1(\mathbb{A}_K)$ since the result of {\rm Arthur-Clozel} \cite{basechange}.

In other words, we speculate that achieving results beyond solvable extensions in both directions (Analyzing gauss sum and proving the automorphy of induction) presents an equal level of difficulty. Proving these results only by the Gauss sum argument for general $K$ would be a challenging task.
\end{remark}

\begin{remark}
Similar assumptions on $K$ appear in the context of the universality of Artin $L$-functions. By using the decomposition of the Dedekind $\zeta$-function of $K$ into the Riemann $\zeta$-function and the Artin $L$-function, {\rm Cho-Kim} \cite{artin} proved the universality of the Artin $L$-function $L_K(s,\rho)$ when $K$ is a $S_{d+1}$-field for $d+1=3,4$, and $5$ (with the additional assumption for $d+1=5$ that $L_K(s,\rho)$ is automorphic).
\end{remark}

\section{Hecke {\it L}-functions and its {\it L}-vaues}\label{sec:cyc:def}
In this section, we give basic definitions and summarize the results in Jun-Lee-Sun \cite{JLS}. 

\subsection{Notation and definition}

Let us list notations and definitions that are frequently used in the paper.

\begin{itemize}

\item For an integer $M\geq 1$, the set of the $M$-th roots of unity in $\FC$ is denoted by $\mu_M$.

\item For an odd prime $p$, the torsion part of $\FZ_p^\times$ is denoted by $\mu_{p-1}$. Hence, we have the decomposition $\FZ_p^\times=\mu_{p-1}\cdot (1+p\FZ_p)$.

\item For $x\in\FQ$, we set $\be(x):=\exp(2\pi ix)$.

\item The number field $K$ is totally real and of degree $d=[K:\FQ]$. The ring of integers in $K$ is denoted by $O_K$ with the different $\mathfrak{d}_K$ and the discriminant $d_K=N_{K/\FQ}(\mathfrak{d}_K)$. The set of totally positive units of $O_K$ is denoted by $E_K$.

\item In this paper, $p>2$ is unramified in $K$.

\item Every ray class group ${\rm Cl}_K(\fm)$ for an integral ideal $\fm$ is defined in the strict sense, i.e., quotient of the set of ideals prime to $\fm$ by the set of totally positive elements that congruent to $1$ modulo $\fm$.
 
\item The Minkowski space $K\otimes_\FQ \FR$ is denoted by $K_\FR$. The set of totally positive elements in $K_\FR$ is denoted by $K_\FR^+$. For any subset $B\subseteq K_\FR$, we set $B^+:=B\cap K_\FR^+$.

%\item A fundamental domain by the action of $E_K$ on $K_\FR^+$ is denoted by $C_K$.

\item We set $K_p:=K\otimes \mathbb{Q}_p$ and $O_p:=O_K\otimes \mathbb{Z}_p$. By C.R.T., we have $O_p=\prod_{\wp|p}O_\wp$ for a prime $\wp$ above $p$ and the integer ring $O_\wp$ of $K_\wp$. 

\item Let $N:O_p\rightarrow \FZ_p$ be the extension of $N_{K/\FQ}$. We use the same symbol for $N:{\rm Cl}_K(p^\infty)\rightarrow \FZ_p^\times$ and $N:K_\FR^\times\rightarrow \FR^\times$.

\item Let $\Tr=\sum_{\wp|p}\Tr_{K_\wp/\mathbb{Q}_p}$ be the %non-degenerate  
trace of $K_p/\mathbb{Q}_p$. %As we assume $p$ is unramified in $K$, the trace $\Tr$ is a perfect pairing on $O_p$ (i.e. $\Tr(\alpha\beta) = 1$ for some $\alpha, \beta \in O_p$).

%\item Let us set $\mathcal{E} :=  \{\alpha \in O_p^\times | N(\alpha) = 1\}$ and for a positive integer $m$, set $\mathcal{E}_m:=\mathcal{E}\cap (1+p^mO_p)$. Their $\FZ_p$-ranks equal $d-1$.
\end{itemize}

\subsection{Ray class characters of totally real fields}\label{sec:ray:class:char}

Let $\xi$ be a ray class character modulo $\fm$. 
Define $\xi_\infty : K_\FR^\times \to S^1$ as the composition 
$$\xi_\infty:K_\FR^\times\rightarrow K_\FR^\times/K_\FR^+\stackrel{\simeq}{\longrightarrow} K_\fm/K_\fm^+\hookrightarrow {\rm Cl}_K(\fm)\stackrel{\xi}{\longrightarrow}S^1,$$
where $K_\fm=\{\alpha\in K\mid \alpha\equiv1\Pmod{\fm}\}$. 
Define $\xi_{\rm f}:(O_K/\fm)^\times \to S^1$ such that   for $\alpha\in O_K$, we set 
\begin{align}\label{def:xif}
	\xi_{\rm f}(\alpha):=\xi(\alpha O_K)\xi_\infty(\alpha)^{-1}.
\end{align}
In particular, we have $\xi(\alpha O_K)=\xi_{\rm f}(\alpha)$  for $\alpha\in K^+$ with $(\alpha,\fm)=1.$

\subsection{Gauss sums of ray class characters}
Let $\xi$ be a ray class character modulo $\mathfrak{m}$.
Let $\beta$ be an element of $\fm^{-1}\mathfrak{d}_K^{-1}\cap K^+$ with $(\beta\fm\mathfrak{d}_K,\fm)=1$. 
The Gauss sum of $\xi$ is defined as
\begin{align*}
\tau(\xi):=\xi(\beta\fm\mathfrak{d}_K)\sum_{\alpha \in (O_K/\fm)^\times}\xi_{\rm f}(\alpha)\be(\Tr(\alpha \beta))
\end{align*}
The condition $\beta\in K^+$ is required to ensure that the definition is independent of the choice of $\beta$. When $\xi$ is primitive, the Gauss sum satisfies
$$|\tau(\xi)|^2=N(\fm)\mbox{ and }\tau(\xi)\tau(\overline{\xi})=\xi_{\rm f}(-1)N(\fm).$$

\subsection{Functional equation of Hecke $L$-functions of totally real fields}\label{sec:hecke:Lvalues}

The Hecke $L$-function for a ray class character $\xi$ is the meromorphic continuation of an $L$-series given by
$$
L_K(s,\xi):=\sum_{\fa }\frac{\xi(\fa)}{N(\fa)^{s}}
$$
for $\Re(s)>1$, where $\fa$ runs through the integral ideals of $K$. Note that we conventionally set $\xi(\fa)=0$ if $\fa$ is not coprime to $\fm$. Let $d_1, d_2$ be the number of $1$, $-1$ in the signature of $\xi$, respectively.
Then the infinite part is given as 
$$L_{\infty}(s,\xi):=\pi^{-\frac{d_1s}{2}}\Gamma\lc\frac{s}2\rc^{d_1}\pi^{-\frac{d_2(s+1)}{2}}\Gamma\lc\frac{s+1}2\rc^{d_2}.$$
The complete Hecke $L$-function is the product
$$\Lambda_K(s,\xi)=d_K^{\frac{s}{2}}N(\fm)^{\frac{s}{2}} L_K(s,\xi) L_{\infty}(s,\xi).$$
The function $\Lambda_K(s,\xi)$ continues analytically to the whole $\FC$ and satisfies the following functional equation 
$$\Lambda_K(s,\xi)=W(\xi)\Lambda_K(1-s,\overline{\xi}  )$$
where
\begin{align*}
W(\xi)&:=i^{-d_2}\frac{\tau(\xi)}{\sqrt{N(\fm)}} 
\end{align*}
(cf. Theorem 8.5 in Neukirch \cite{Neu}).
If $\xi$ is totally odd, then we obtain that
\begin{equation}\label{algebraicity}
L_K(0,\overline{\xi})=\frac{\sqrt{d_K}}{(\pi i)^d}\tau(\overline{\xi})L_K(1,\xi)
\end{equation}
by putting $s=1$. 
As the partial zeta value for a ray class at the critical point $0$ is a rational number (see Neukirch \cite[Theorem 9.8]{Neu}), we have $L_K(0,\xi)\in \FQ(\xi)$ and for $\sigma \in \Gal(\overline{\FQ}/\FQ),$ 
\begin{equation}\label{reci}
L_K(0,\xi)^\sigma=L_K(0,\xi^{\sigma}).
\end{equation}
From (\ref{algebraicity}) and (\ref{reci}), we have the following reciprocity law:
\begin{prop}\label{reciprocity}
Let $\sigma\in{\rm \Gal}(\overline{\FQ}/\FQ)$ and $\xi$ be a totally odd ray class character. Then one has
\begin{align*}
\left(\frac{\sqrt{d_K}\tau(\overline{\xi}  )L_K(1,\xi)}{(\pi i)^{d}}\right)^\sigma= \frac{\sqrt{d_K}\tau(\overline{\xi}^\sigma)L_K(1,\xi^\sigma)}{(\pi i)^{d}}.
\end{align*}
\end{prop}

\subsection{Cyclotomic characters over totally real fields}\label{sec:cyclotomic_ch}%\label{cycl_hecke}

In this section, we present a definition of cyclotomic characters, and discuss the explicit expressions of Gauss sums.

\begin{defn}
A {\it cyclotomic character} over $K$ modulo a $p$-power is a character $\chi$ on ${\rm Cl}_K(p^{\infty})$ that is a composition of the norm $N$ and a Dirichlet character $\psi$ on $\FZ_p^\times$:
$$
\chi:{\rm Cl}_K(p^\infty)\stackrel{N}{\longrightarrow}\FZ_p^\times\stackrel{\psi}{\longrightarrow}\FC^\times.
$$
\end{defn}

Note that $\psi_\infty:\FR^\times\rightarrow S^1$ is given as $\psi_\infty=\psi\circ {\rm sgn}$ where ${\rm sgn}:\FR^\times\rightarrow \{\pm1\}$ is the sign function. The cyclotomic characters satisfying the following relations:
\begin{prop}[Jun-Lee-Sun {\cite[Proposition 3.1]{JLS}}]\label{cyclo}
Let $\chi=\psi\circ N$ be a cyclotomic character on ${\rm Cl}_K(p^n)$. Then, the followings are true:
\begin{enumerate}
\item $\chi_{\rm f}=\psi\circ N$ and $\chi_\infty=\psi_\infty\circ N$.
\item  $\chi$  is primitive if and only if $\psi$ is primitive.
\item $\chi$ is totally even if and only if $\psi$ is even.
\end{enumerate}
\end{prop}

For a cyclotomic character $\chi$ modulo $p^n$, we obtain
$$\tau(\chi)=\chi(\beta p^n\mathfrak{d}_K)\sum_{\alpha\in (O_K/p^nO_K)^\times} \chif(\alpha)\be(\Tr(\alpha \beta))$$
for any $\beta\in p^{-n}\mathfrak{d}_K^{-1}\cap K^+$ with $(\beta p^n\mathfrak{d}_K,p)=1$. Since $p$ is unramified in $K$, we may set $\beta=\frac{1}{p^n}$ and get
$$\tau(\chi)=\chi(\mathfrak{d}_K)\sum_{\alpha\in (O_K/p^nO_K)^\times}\chif(\alpha)\be\lc\frac{\Tr(\alpha)}{p^n}\rc.$$

\subsection{Averages of $L$-values}\label{sec:7}
  
In this subsection, we study averages of the traces of $L$-values after investigating the averages of the characters and Gauss sums.

Let $\mathfrak{m}$ be an integral ideal of $K$.
Let $\eta$ be a totally odd primitive ray class character over $K$ modulo $\mathfrak{m}$.
Let $n_0>1$ be the integer such that 
$$
\mathbb{Q}(\mu_{p^2},\eta)\cap \mathbb{Q}(\chi)=\mathbb{Q}(\mu_{p^{n_0}}).
$$
Let us set $F=\mathbb{Q}(\mu_{p^{n_0}(p-1)})$. In this section, let us set $\chi$ by a totally even primitive cyclotomic character over $K$ modulo $p^n$ and 
$$
G:=\Gal(F(\chi)/F).
$$ 
Due to Proposition \ref{cyclo}, we obtain $\chif= \psi \circ N$ for a totally even primitive Dirichlet character $\psi$ modulo $p^n$. Let $\alpha\in O_p$.
If $n>2n_0+1$, then we have $\psi(N(1+ p^{n-n_0} \alpha)) = 
\psi(1+p^{n-n_0} \Tr (\alpha))$. Hence we have
\begin{equation}\label{def:c_chi}\begin{split}
\chif(1+p^{n-n_0} \alpha) = 
\be\lc \frac{b_\chi \Tr(\alpha)}{p^{n_0} }\rc \mbox{ for a }b_\chi\in (\FZ/p^{n_0} \FZ)^\times \mbox{ and }n>3.
\end{split}
\end{equation}

We define  a {\it  partial Gauss sum} of $\chi$ as
\begin{align*}
\tau_1(\chi):=\chi(\mathfrak{d}_K )\sum_{\substack{\delta\in (O_K/p^nO_K)^\times \\ \delta \equiv  -b_{\chi}(p^{n_0})}} \chif(\delta)\be\lc \frac{\Tr( \delta)}{p^{n}}\rc .
\end{align*}
For an ideal $\fa$ prime to $p$, we define the two types of averages of Gauss sums by 
\begin{align*}
 \tau_{\rm av}(\chi,\fa) &:=\frac{1}{|G|}\sum_{\sigma\in G} \tau_1(\chi^{\sigma})\tau(\overline{\chi}^{\sigma})\chi^{\sigma}(\fa),\\
\widetilde{\tau}_{\rm av}(\chi,\fa) &:=\frac{1}{|G|}\sum_{\sigma \in G}\tau_1(\chi^{\sigma})\overline{\chi}^{\sigma}(\fa).
\end{align*}
Then, we have an explicit asymptotic for an average of $L$-values $L_K(0,\overline{\eta\chi})$:

\begin{thm}\label{theorem2}
Let $\mathfrak{r}$ be a fractional ideal of $K$ whose norm is coprime to $p$, then
\begin{equation}
\label{oe}
\begin{split}
\frac{1}{|G|N(p^n)}&\sum_{\delta\equiv -b_{\chi} (p^{n_0})}\be\lc\frac{\Tr(\delta)}{p^n }\rc\Tr_{F(\eta,\chi)/F(\eta)}\Big{(}\chi(\delta \mathfrak{r}^{-1} )L_K(0,\overline{\eta\chi})\Big{)}
\\& =\frac{\sqrt{d_K}}{(\pi i)^d N(\mathfrak{r})}\sum_{N(\mathfrak{a})=N(\mathfrak{r})}\eta(\mathfrak{a})+o(1)\text{ as }n\rightarrow\infty.
\end{split}
\end{equation}
\end{thm}

\begin{proof} 
Using the functional equation (\ref{algebraicity}), one can replace $L_K(0,\overline{\eta\chi})$ in (\ref{oe}) with 
$$
L_K(0,\overline{\eta\chi})=\frac{\sqrt{d_K}}{(\pi i)^d}\tau(\overline{\eta\chi}  )L_K(1,\eta\chi).
$$ 
Moreover, using Proposition \ref{reciprocity} and the appoximate functional equation (Jun-Lee-Sun \cite[Theorem 2.3]{JLS}), one can write the L.H.S. of \eqref{oe} as
\begin{align*}
&\frac{\sqrt{d_K}}{(\pi i)^d}\frac{1}{N(p^n)}\sum_{\fa} \frac{\eta(\fa)\tau_{\rm av}(\chi,\fa\mathfrak{r}^{-1})}{N(\fa)}F_1\lc\frac{N(\fa)}{y}\rc \\ 
&-\frac{1}{\pi^d N(p^n)}\sum_{\fa}\overline{\eta}(\fa)\widetilde{\tau}_{\rm av}(\chi,\fa \mathfrak{r} )F_2\lc\frac{N(\fa)y}{d_KN(p)^n}\rc.
\end{align*}
As $|\eta|\leq 1$, the estimation on the error terms in the above equations is consistent with that of the average of $L$-values in Jun-Lee-Sun \cite[Theorem 7.5]{JLS}. On the other hand, we observe from the property of the average $\tau_{\rm av}$ and the estimation of a smooth function $F_1$ (see Jun-Lee-Sun \cite[Proposition 7.3, equation (18)]{JLS}) that the main term is given by the norm average:
\begin{equation}\label{congruence:changed}
\frac{\sqrt{d_K}}{(\pi i)^d}\frac{1}{N(\mathfrak{r})}\sum_{N(\mathfrak{a})=N(\mathfrak{r})}\eta(\mathfrak{a}).
\end{equation}
Therefore, we conclude that
\begin{align*}
\frac{1}{|G|N(p^n)}&\sum_{\delta\equiv -b_{\chi} (p^{n_0})}\be\lc\Tr\left(\frac{\delta}{p^n }\right)\rc\Tr_{F(\chi)/F}\Big{(}\chi(\delta \mathfrak{r}^{-1} )L(0,\overline{\chi})\Big{)}\\
& =\frac{\sqrt{d_K}}{(\pi i)^d}
	\frac{ 1 }{{N( \mathfrak{r})}}\sum_{N(\fa)=N(\mathfrak{r})}\eta(\fa)
	+o(1)
	\end{align*}
	as $n\rightarrow \infty$.
	\end{proof}

\begin{remark}
In {\rm Jun-Lee-Sun} \cite{JLS}, $n_0$ is defined by the integer satisfying
$$
\mathbb{Q}(\mu_{p},\eta)\cap \mathbb{Q}(\chi)=\mathbb{Q}(\mu_{p^{n_0}}).
$$
However, under this definition, the average for $\delta\equiv -b_\chi\ (p^{n_0})$ cannot be applied to the proof of {\rm Proposition \ref{prim}}, which is our critical result described in the following section, due to the definition of primitive elements. Thus, we have corrected slightly the definition as follows:
$n_0>1$ is the integer such that
$$
\mathbb{Q}(\mu_{p^2},\eta)\cap \mathbb{Q}(\chi)=\mathbb{Q}(\mu_{p^{n_0}}).
$$
Additionally, we modify the coefficient in the main term of the average in {\rm Jun-Lee-Sun} \cite[Theorem 7.5]{JLS} and record it in {\rm Theorem \ref{theorem2}} because there is a minor error in the reciprocity result {\rm (Jun-Lee-Sun \cite[equation (16)]{JLS})}, which is revised in {\rm Proposition \ref{reciprocity}} of this paper. We also change the definition of the partial Gauss sum $\tau_1$ slightly so that $d_K$ does not appear in the congruence conditions in {\rm equation (\ref{congruence:changed})} in contrast with that of {\rm Jun-Lee-Sun} \cite[Theorem 7.5]{JLS}.
\end{remark}

\section{Generation by cyclotomic Hecke {\it L}-values }\label{final}
In this section, we prove that the field generated by a single cyclotomic Hecke $L$-value contains norm averages of a tame character $\eta$.
Here recall that $p>2$ is unramified in $K$, $\eta$ is a totally odd primitive ray class character over $K$ modulo $\mathfrak{m}$, and $F=\FQ(\mu_{p^{n_0}(p-1)})$.
\begin{defn}
Let us define the norm average function $c_\eta:\mathbb{Z}_{>0}\to\mathbb{Z}[\eta]$ attached to $\eta$ as follows:
$$
c_\eta(m):=\sum_{ N(\mathfrak{a})=m }  \eta(\mathfrak{a}).
$$
Let us provide the following notations for simplicity:
\begin{itemize}
\item[(1)] Let us denote by $c^\prime_\eta$ the restriction of the function $c_\eta$ to the set of positive integers coprime to $p$.
\item[(2)] Let us denote by $\Xi_{K,p}$ the set of totally even primitive cyclotomic characters over $K$ modulo a $p$-power.
\end{itemize}
\end{defn}
Based on the estimation of Theorem \ref{theorem2}, we obtain the following generation result:
\begin{prop}\label{hecke:field:generation}
For almost all $\chi\in\Xi_{K,p}$,
$$F(L_K(0,\eta\chi),\chi)=F(c_\eta^\prime,\chi).$$
\end{prop}

\begin{proof}
Let $\chi\in\Xi_{K,p}$ of modulus $p^n$ with sufficiently large $n$.
Let us denote by $L_\chi:=F\big(L_K(0,\eta\chi)\big)$ in this proof.
If $\sigma\in\mathrm{Gal}\big(F(\eta,\chi)/F(c_\eta^\prime,\chi)\big)$, then we obtain from the reciprocity (\ref{reci}) that
$$
L_K(0,\eta\chi)^\sigma=L_K(0,\eta^\sigma\chi)=L_K(s,\eta^\sigma\chi)|_{s=0}=L_K(s,\eta\chi)|_{s=0}=L_K(0,\eta\chi),
$$
where the third equality comes from the facts that
$$
L_K(s,\eta\chi)=\sum_{m>0}\frac{c_\eta(m)\psi(m)}{m^s}
$$
and $c_\eta^\prime=(c_{\eta}^\prime)^\sigma=c_{\eta^\sigma}^\prime$.
From this, we conclude that $L_\chi\subset F(c_\eta^\prime,\chi)$ since $\sigma$ is arbitrary.

For a normal subgroup $H$ of $\mathrm{Gal}(F(\eta)/F)$, let us denote by
$$
S_H:=\{\chi\in\Xi_{K,p}\mid L_\chi(\chi)\cap F(\eta)=F(\eta)^H \}\subset \Xi_{K,p}.
$$ 
Due to the infiniteness of the set $\Xi_{K,p}$ and the finiteness of the dimension of $F(\eta)$ over $F$, there exists a normal subgroup $H$ of $\mathrm{Gal}(F(\eta)/F)$ such that $S_H$ is infinite. 
For $\chi\in S_H$, let $\sigma$ be an element of the Galois group $\mathrm{Gal}\big(F(\eta,\chi)/L_\chi(\chi)\big)$, which is isomorphic to $H$ as $L_\chi(\chi)F(\eta)=F(\eta,\chi)$, so $\mathrm{Gal}\big(F(\eta,\chi)/L_\chi(\chi)\big)$ is independent on $\chi$. 
Then, by the reciprocity (\ref{reci}), we have
\begin{equation}\label{cyclotomic:proof:equation}
L_K(0,\eta\chi)=L_K(0,\eta\chi)^\sigma=L_K(0,\eta^\sigma\chi)
\end{equation}
for $\chi\in S_H$.
Note that $F(\eta)=F(\eta^\sigma)$.
Thus, by taking the average to equation (\ref{cyclotomic:proof:equation}), taking the limit for $\chi\in S_H$ (which is possible as $\Xi_{K,p}$ is a line of characters and $|S_H|=\infty$) and applying Theorem \ref{theorem2}, we prove that
$$
c_\eta^\prime=c_{\eta^\sigma}^\prime=(c_{\eta}^\prime)^\sigma.
$$
Since $\sigma$ is arbitrary, we conclude that $F(c_\eta^\prime)\subset L_\chi(\chi)$. So we are done.
\end{proof}

\begin{defn}
We call an ideal $\mathfrak{r}$ of $K$ {\it primitive} if $\langle N(\mathfrak{r}) \rangle\neq 1$. 
\end{defn}
If $\mathfrak{r}$ is primitive, then for any $\chi\in\Xi_{K,p}$ of modulus $p^n$ with $n>n_0$, we have 
$$
F\big(\chi(\mathfrak{r})\big)=F(\chi).
$$
We have the following crucial non-vanishing lemma, whose proof is provided in the next section:
\begin{lem}\label{nonvanishing:primitive}
If $K$ is solvable and $p$ does not divide $N(\mathfrak{m})$, then $c_\eta(m)\neq 0$ for some integer $m$ with $\langle m\rangle\neq 1$.
\end{lem}
By using this important lemma, we obtain the following non-vanishing result:

\begin{prop} \label{prim}
Let us assume that $K$ is solvable and $p$ does not divide $N(\mathfrak{m})$.
Then, there exists a primitive element $\mathfrak{r}$ such that
$$
\Tr_{F(\eta,\chi)/F(\eta)}\Big{(}\chi(\mathfrak{r})L_K(0,\overline{\eta\chi})\Big{)}\not = 0
$$ 
for almost all $\chi\in\Xi_{K,p}$. 
\end{prop}

\begin{proof}
Let $\chi\in\Xi_{K,p}$ of modulus $p^n$. Note that both taking the inverse and the multiplication by an imprimitive element give permutations on the set of primitive elements. Thus, if $b_\chi$ is imprimitive, then there is a primitive element $\mathfrak{r}_0$ such that $c_\eta\big(N(b_\chi\mathfrak{r}_0^{-1})\big)\neq 0$ by Lemma \ref{nonvanishing:primitive}. 
Note that this lemma needs the solvability of $K$ and the $p$-indivisibility of $N(\mathfrak{m})$.
Choose $\mathfrak{r}_1$ such that
$$\mathfrak{r}_1=\begin{cases}O_K&\text{if $b_{\chi}$\text{ is primitive,}}\\ b_\chi\mathfrak{r}_0^{-1} &\text{otherwise} \end{cases}.$$
Then, $\delta\mathfrak{r}_1^{-1}$ is a primitive element for any $\delta\equiv -b_\chi\ (p^{n_0})$ as $n_0>1$.
Due to Theorem \ref{theorem2}, we have 
\begin{equation*}
\begin{split}
\frac{1}{|G| N(p)^{n}}&\sum_{\delta \equiv -b_{\chi}(p^{n_0})}\be\lc \Tr\lc\frac{\delta  }{p^{n}}\rc \rc 
\Tr_{F(\eta,\chi)/F(\eta)}\Big{(}\chi(\delta\mathfrak{r}_1^{-1})L_K(0,\overline{\eta\chi})\Big{)}\\
&= \frac{\sqrt{d_K}}{(\pi i)^d N(\mathfrak{r}_1)}c_\eta\big(N(\mathfrak{r}_1)\big)+o(1) .
\end{split}
\end{equation*}
Note that the main term of the last expression is nonzero due to our choice on $\mathfrak{r}_1$ and the fact that $c_\eta(1)=\eta(O_K)=1$. Thus, if $n$ is sufficiently large then there is at least one $\delta\equiv -b_{\chi}\ (p^{n_0})$ such that 
$$\Tr_{F(\eta,\chi)/F(\eta)}\Big{(}\chi(\delta\mathfrak{r}_1^{-1})L_K(0,\overline{\eta\chi}) \rc\not=0.$$
This completes the proof.
\end{proof} 

From this, we observe that a critical $L$-value generates a cyclotomic field $F(\chi)$:
\begin{prop}\label{cyclotomic:generation} 
Let us assume that $K$ is solvable and $p$ does not divide $N(\mathfrak{m})$.
For almost all $\chi\in\Xi_{K,p}$,
$$
F(\chi)\subset F\big(L_K(0,\eta\chi)\big).
$$
\end{prop}
\begin{proof}
Let us denote by 
$
L_\chi:=F\big(L_K(0,\eta\chi)\big)
$
and
$$
A_\chi(\mathfrak{r}):=\Tr_{F(\eta,\chi)/F(\eta)}\Big{(}\chi(\mathfrak{r})L_K(0,\overline{\eta\chi})\Big{)}
$$
for $\chi\in\Xi_{K,p}$ in this proof. Note that there is a primitive element $\mathfrak{r}$ such that 
$$
A_\chi(\mathfrak{r})\neq 0
$$
for almost all $\chi\in\Xi_{K,p}$ by Proposition \ref{prim}.
%Clearly, one can choose $\alpha\in F(\eta)$ such that $F(c_\eta^\prime,\alpha)=F(\eta)$.
We can rewrite $A_\chi(\mathfrak{r})$ as follows:
$$
A_\chi(\mathfrak{r})=\Tr_{F(\chi)L_\chi(\eta)/F(\eta)}\Big{(}\chi(\mathfrak{r})L_K(0,\overline{\eta\chi})\Big{)}.
$$
Let us assume the contrary that $L_\chi(\eta)\cap F(\chi)\subsetneq F(\chi)$ for infinitely many $\chi$. Let us denote by $S$ the infinite set of such $\chi$'s. From the transitivity of the trace and the fact that $F(\eta)\subset L_\chi(\eta)\subset F(\chi)L_\chi(\eta)$, we may express $A_\chi(\mathfrak{r})$ as 
\begin{equation*}
A_\chi(\mathfrak{r})=\Tr_{L_\chi(\eta)/F(\eta)}\lc  L_K(0,\overline{\eta\chi})\Tr_{F(\chi)L_\chi(\eta)/L_\chi(\eta)}\big(\chi(\mathfrak{r}) \big)\rc.
\end{equation*}
Since $\mathfrak{r}$ is primitive and $L_\chi(\eta)\cap F(\chi)\subsetneq F(\chi)$ for $\chi\in S$,
$$
\Tr_{F(\chi)L_\chi(\eta)/L_\chi(\eta)}\big(\chi(\mathfrak{r})\big)=\Tr_{F(\chi)/L_\chi(\eta)\cap F(\chi)}\big(\chi(\mathfrak{r})\big)=0
$$
for $\chi\in S$. Therefore, $A_\chi(\mathfrak{r})=0$ for infinitely many $\chi$, which is a contradiction. Thus, we obtain that $F(\chi)\subset L_\chi(\eta)$ for almost all $\chi\in\Xi_{K,p}$.  Note that $F(\eta)=F(\mu_M)$ with $p\nmid M$ since $\eta$ has finite order and $F$ contains all the $p$-power roots of unity of $\eta$. Therefore, 
$$
F(\chi)\subset L_\chi(\eta)\cap F(\chi)=L_\chi(\mu_M)\cap F(\chi)=L_\chi
$$ 
for almost all $\chi\in\Xi_{K,p}$.
So we are done.

%Let $T$ be the set of $\chi\in\Xi_{K,p}$ such that $F(\chi)\subset L_\chi(\alpha)$, so that the set $\Xi_{K,p}-T$ is of finite and $L_\chi(\alpha)=F(\eta,\chi)$ for $\chi\in T$. Let us assume the contrary that $F(\chi)\not\subset L_\chi$ for some $\chi\in T$ of modulus $p^n$.  Note that $F(\eta)=F(\mu_M)$ with $p\nmid M$ since $\eta$ has finite order and $F$ contains all the $p$-power roots of unity of $\eta$. Thus, $L_\chi(\chi)\subset L_\chi(\mu_M)$, however due to our assumption on $L_\chi$ and $M$,
%$$
%\mathbb{Q}(\mu_{p^{n-1}})=L_\chi(\chi)\cap\mathbb{Q}(\mu_{p^{n-1}})\subset L_\chi(\mu_M)\cap \mathbb{Q}(\mu_{p^{n-1}})\subsetneq\mathbb{Q}(\mu_{p^{n-1}}),
%$$
%which is a contradiction. Thus, $F(\chi)\subset L_\chi$. So we are done.
\end{proof}

We can conclude that a single critical $L$-value $L_K(0,\eta\chi)$ generates $F(c_\eta^\prime,\chi)$:
\begin{thm}\label{main:result}
%Let $p$ be a rational prime unramified in $K$. Let $\mathfrak{m}$ be an integral ideal of $K$ coprime to $p$. Let $\eta$ be a totally odd primitive character over $K$ modulo $\mathfrak{m}$.
Let us assume that $K$ is solvable and $p$ does not divide $N(\mathfrak{m})$. Then,
$$
F\big(L_K(0,\eta\chi)\big)=F(c^\prime_\eta,\chi)
$$
for almost all $\chi\in\Xi_{K,p}$.
\end{thm}

\begin{proof}
Immediate by Proposition \ref{hecke:field:generation} and Proposition \ref{cyclotomic:generation}.
\end{proof}

\section{Hecke field generation by norm average of Hecke characters}
Let us recall that $\eta$ is a totally odd primitive ray class character $\eta$ over $K$ modulo $\mathfrak{m}$.
Let us suggest the following question:
\begin{question*} Does the following equality hold?
$$
\mathbb{Q}(c_\eta^\prime)=\mathbb{Q}(\eta).
$$
\end{question*}
From Theorem \ref{main:result}, we observe that Question JLS$^\prime$ implies Question JLS under additional assumptions on $\mathfrak{m}$ and $K$. 
In this section, we prove a result toward this question.

\subsection{Gauss sums of characters over solvable number fields}
In this section, we decompose the Gauss sum of cyclotomic characters over solvable $K$.

Firstly, we give an easy but useful lemma:
\begin{lem}\label{gausssum:decomposition}
Let $\phi$ be a ray class character over $K$ modulo $\mathfrak{m}$. 
Let $M\nmid \mathfrak{m}$ be a positive integer. 
Let $\chi$ be a ray class character over $K$ modulo $MO_K$. Then, $\chi\phi$ is a ray class character over $K$ modulo $M\mathfrak{m}$ and 
$$
\tau(\chi\phi)=\tau(\chi)\tau(\phi)\chi(\mathfrak{m})\phi_{\rm f}(M).
$$
\end{lem}

\begin{proof}
Immediate from the definition of Gauss sums and straight computations.
%Let $\beta$ be an element of $(M\mathfrak{m})^{-1}\mathfrak{d}_K^{-1}\cap K^+$ with $(\beta M\mathfrak{m}\mathfrak{d}_K,M\mathfrak{m})=1$. Then clearly
%$\beta M\in \mathfrak{m}^{-1}\mathfrak{d}_K^{-1}\cap K^+$ and $(\beta M\mathfrak{m}\mathfrak{d}_K,\mathfrak{m})=1$.
%Thus, we can choose elements $x\in O_K$ and $y\in\mathfrak{m}$ such that $Mx+y=1$. From this, we can check that the following map is an isomorphism: 
%$$
%(O_K/MO_K)^\times\times(O_K/\mathfrak{m})^\times\to(O_K/M\mathfrak{m})^\times,\ (\alpha_1,\alpha_2)\mapsto \alpha_1 y+\alpha_2 Mx.
%$$  
%Note that 
%\begin{align*}
%&\beta Mx\in \mathfrak{m}^{-1}\mathfrak{d}_K^{-1},\ \beta y\in M^{-1}\mathfrak{d}_K^{-1}, \\
%&(\beta Mx\cdot\mathfrak{m}\mathfrak{d}_K,\mathfrak{m})=(x,\mathfrak{m})=(Mx,\mathfrak{m})=(1-y,\mathfrak{m})=1, \\
%&(\beta y\cdot M\mathfrak{d}_K,M)=(y,M)=(1-Mx,M)=1.
%\end{align*}
%From this, we obtain that
%\begin{align*}
%\tau(\chi\phi)=&\chi\phi(\beta M\mathfrak{m}\mathfrak{d}_K)
%\sum_{\alpha\in (O_K/M\mathfrak{m})^\times}(\chi\phi)_{\rm f}(\alpha)\mathbf{e}\big({\rm Tr}(\alpha\beta)\big) \\
%=&\chi(\beta M\mathfrak{m}\mathfrak{d}_K)
%\sum_{\alpha_1\in (O_K/MO_K)^\times}\chi_{\rm f}(\alpha_1 y)\mathbf{e}\big({\rm Tr}(\alpha_1 \beta y)\big) \\
%\times&\phi(\beta M\mathfrak{m}\mathfrak{d}_K)\sum_{\alpha_2\in (O_K/\mathfrak{m})^\times}\phi_{\rm f}(\alpha_2 Mx)\mathbf{e}\big({\rm Tr}(\alpha_2\beta Mx)\big) \\
%=&\tau(\chi)\tau(\phi)\chi(\mathfrak{m})\phi_{\rm f}(M).
%\end{align*}
%So we are done.
\end{proof}
Let us assume that $K$ is solvable, i.e., $K$ is a Galois number field and its Galois group $\mathrm{Gal}(K/\mathbb{Q})$ is solvable.
Then, there is a tower of field extensions 
$$
K=K_r\supset K_{r-1}\supset\cdots\supset K_1\supset K_0=\mathbb{Q}
$$ 
such that $K_{i}/K_{i-1}$ is abelian for each $i=1,\cdots,r$.
Under this situation, there is a decomposition of the Gauss sum of cyclotomic characters over $K$:

\begin{prop}\label{gausssum:decomposition:2}
Let $\chi=\psi\circ N$ be a cyclotomic character over $K$ modulo $p^n$. Then,
$$
\tau(\chi)=C_n\psi(d_K)\tau(\psi)^{d}
$$ 
for some non-zero number $C_n$ depends only on $n$ and $K$.
\end{prop}
\begin{proof}
For simplicity, let us provide the following notations in this proof: 
Let $G_i$ be the Galois group $\mathrm{Gal}(K_i/K_{i-1})$, $G_i^*$ the group of primitive ray class characters over $K_{i-1}$ associated to the dual of $G_i$.
Let us denote by $\chi_i=\psi\circ N_{K_i/\mathbb{Q}}$, which is a primitive cyclotomic character over $K_i$ modulo $p^n$ by Proposition \ref{cyclo}.
Let $d_i$ be the relative discriminant of $K_i/K_{i-1}$, which is an ideal of $K_{i-1}$. 
%Let us denote $N(\mathfrak{a})=N_{K_i/\mathbb{Q}}(\mathfrak{a})$ for an ideal $\mathfrak{a}$ of $K_i$. 
Note that 
$$
d_{K_i}=N_{K_{i-1}/\mathbb{Q}}(d_i)d_{K_{i-1}}^{|G_i|}.
$$
We keep using this relation in the proof.

From the global class field theory, we observe the following decomposition of $L$-function:
$$
L_{K_i}(s,\chi_i)=\prod_{\phi\in G_i^*}L_{K_{i-1}}(s,\chi_{i-1}\phi)
$$
since $\chi_i$ factors through as $\chi_i=\chi_{i-1}\circ N_{K_i/K_{i-1}}$.
Note that the characters $\phi\in G_i^*$ are totally even since $K_{i-1}$ is totally real, thus, 
$$
L_\infty(s,\chi_i\phi)=\prod_{\phi\in G_i^*}L_\infty(s,\chi_{i-1}\phi)
$$ 
by Proposition \ref{cyclo}. From this, we obtain the following equality:
\begin{equation}\label{complete:Lfunction:decomposition}
\Lambda_{K_i}(s,\chi_i)=N_{K_{i-1}/\mathbb{Q}}(d_i)^{\frac{s}{2}}\prod_{\phi\in G_i^*}N_{K_{i-1}/\mathbb{Q}}(\mathfrak{f}_\phi)^{-\frac{s}{2}}\Lambda_{K_{i-1}}(s,\chi_{i-1}\phi),
\end{equation}
where $\mathfrak{f}_\phi$ is the conductor of $\phi$. 
By the functional equation of $L_K$,
\begin{align}
\label{func:eq:1}
\Lambda_{K_i}(s,\chi_i)&=\frac{\tau(\chi_i)}{(i^{\delta}p^{n/2})^{|G_i|}}\Lambda_{K_i}(1-s,\overline{\chi_i}),
\\
\label{func:eq:2}
\prod_{\phi\in G_i^*}\Lambda_{K_{i-1}}(s,\chi_{i-1}\phi)&=\prod_{\phi\in G_i^*}\frac{\tau(\chi_{i-1}\phi)\Lambda_{K_{i-1}}(1-s,\overline{\chi_{i-1}\phi})}{(i^{\delta}p^{n/2})^{|G_{i-1}|}N_{K_{i-1}/\mathbb{Q}}(\mathfrak{f}_\phi)^{1/2}},
\end{align} 
where $\delta=(1-\psi(-1))/2$.
Dividing equation (\ref{func:eq:1}) by equation (\ref{func:eq:2}) and applying equation (\ref{complete:Lfunction:decomposition}), we obtain that
$$
N_{K_{i-1}/\mathbb{Q}}(d_i)^{s}\prod_{\phi}N_{K_{i-1}/\mathbb{Q}}(\mathfrak{f}_\phi)^{-s}=N_{K_{i-1}/\mathbb{Q}}(d_i)^{\frac{1}{2}}\tau(\chi_i)\prod_{\phi}\tau(\chi_{i-1}\phi)^{-1}.
$$
Since the right hand side is independent on $s$, the followings hold:
$$
N_{K_{i-1}/\mathbb{Q}}(d_i)=\prod_{\phi}N_{K_{i-1}/\mathbb{Q}}(\mathfrak{f}_\phi), \ \tau(\chi_i)=N_{K_{i-1}/\mathbb{Q}}(d_i)^{-\frac{1}{2}}\prod_{\phi}\tau(\chi_{i-1}\phi).
$$
Applying Lemma \ref{gausssum:decomposition} to the above equation, we obtain that
$$
\tau(\chi_i)=\Big(N_{K_{i-1}/\mathbb{Q}}(d_i)^{-\frac{1}{2}}\psi\big(N_{K_{i-1}/\mathbb{Q}}(d_i)\big)\prod_{\phi}\tau(\phi)\phi_{\rm f}(p^n)\Big)\tau(\chi_{i-1})^{|G_i|}.
$$
Finally, we conclude from the mathematical induction that 
$$
\tau(\chi)=d_K^{-\frac{1}{2}}\psi(d_K)\tau(\psi)^{d}\prod_{i=1}^r\Big(\prod_{\phi\in G_i^*}\tau(\phi)\phi_{\rm f}(p^n)\Big)^{[K:K_{i}]}.
$$
So we are done.
\end{proof}

\subsection{Hyper Kloosterman sums and average of gauss sums}
For integers $d,M>0$ and $r\in (\mathbb{Z}/M\mathbb{Z})^\times$, let us give the definition of the hyper Kloosterman sum of dimension $d$ and of modulus $M$ as follows: 
$$
\mathrm{Kl}_d(r;M):=\sum_{\substack{n_1,\cdots,n_d\in  (\mathbb{Z}/M\mathbb{Z})^\times \\ n_1\times\cdots\times n_d\equiv r\ (M) }}\mathbf{e}\Big(\frac{n_1+\cdots+n_d}{M}\Big).
$$
Gurak \cite{Gurak} computes the explicit values $\mathrm{Kl}_d(r;p^n)$ when $n>1$. We need the following non-vanishing result by Gurak \cite{Gurak}:
\begin{lem}[Gurak {\cite[Lemma 1,2]{Gurak}}]\label{kloosterman:nonvanishing:lemma}
$\mathrm{Kl}_d(r;p^2)$ is non-vanishing if and only if 
$$
r^{\frac{p-1}{(d,p-1)}}\equiv 1\ (p).
$$
\end{lem}

\begin{remark} 
If $d=2$, then $\mathrm{Kl}_2(r,p^2)\neq 0$ if and only if $\big(\frac{r}{p}\big)=1$ by {\rm Lemma \ref{kloosterman:nonvanishing:lemma}}, where $\big(\frac{\cdot}{p}\big)$ is the quadratic residue symbol.
\end{remark}

%Note that $\omega$ and $\psi$ generates the dual of the finite abelian group $(\mathbb{Z}/p^2\mathbb{Z})^\times$. 
By applying the non-vanishing of hyper Kloosterman sums and the decomposition of Gauss sums, we prove the non-vanishing of twisted average of Gauss sums when $K$ is solvable:

\begin{prop}\label{twisted:average:gausssum}
Let $\omega:(\mathbb{Z}/p\mathbb{Z})^\times\rightarrow\mu_{p-1}$ be the Teichm{\" u}ller character. 
Let $\psi$ be a generator of a cyclic group $\mathrm{Hom}\big((\mathbb{Z}/p^2\mathbb{Z})^\times,\mu_p\big)$.
Let $m$ be an integer coprime to $p$. If $K$ is solvable, then
$$
\sum_{i=1}^{p-1}\tau(\psi^i\omega^j\circ N)\psi^i(m)\neq 0
$$
for some $j=1,\cdots,p-1$.
\end{prop}

\begin{proof}
If $p\mid i$, then clearly $\tau(\psi^i\omega^j\circ N)=0$ by Proposition \ref{gausssum:decomposition:2} since $\psi^i\omega^j$ is imprimitive so that $\tau(\psi^i\omega^j)^2=0$. Thus, it is enough to show that $\sum_{i=1}^{p} \tau(\psi^i\omega^j\circ N)\psi^i(m)\neq 0$.  

It is well known that for $\phi\in\widehat{(\mathbb{Z}/p^2\mathbb{Z})^\times}$,
$$
\tau(\phi)^d=\sum_{r=1}^{p^2}\phi(r)\mathrm{Kl}_d(r;p^2).
$$
Let us choose an element $\kappa$ of $\mu_{p-1}$ such that $\kappa md_K\equiv 1 \ (p)$. Note that $\psi(\kappa)=1$.
Then, based on the orthogonality and Proposition \ref{gausssum:decomposition:2}, we verify that
\begin{align*}
\sum_{j=1}^{p-1}\omega^j(\kappa m)\sum_{i=1}^{p}\tau(\psi^i\omega^j\circ N)\psi^i(m)&=C\sum_{i,j}\tau(\psi^i\omega^j)^d \psi^i\omega^j(\kappa m d_K) \\ 
=C\sum_{\phi\in \widehat{(\mathbb{Z}/p^2\mathbb{Z})^\times}}\sum_{r=1}^{p^2} \phi(\kappa m d_K r)\mathrm{Kl}_d(r;p^2)&=Cp(p-1)\mathrm{Kl}_d\big((\kappa m d_K)^{-1};p^2\big)
\end{align*}
for some non-zero number $C$, where $a^{-1}$ is the inverse of $a$ mod $p^2$.
Since $\kappa m d_K \equiv 1\ (p)$, we have $\mathrm{Kl}_d\big((\kappa m d_K)^{-1};p^2\big)\neq 0$ by Lemma \ref{kloosterman:nonvanishing:lemma}, so we are done.
\end{proof}

\subsection{Partial Hecke {\it L}-functions}
Let us recall that $\eta$ is a primitive ray class character over $K$ modulo $\mathfrak{m}$, and the norm average function $c_\eta:\mathbb{Z}_{>0}\to\mathbb{Z}[\eta]$ attached to $\eta$ is given by
$$
c_\eta(m)=\sum_{ N(\mathfrak{a})=m }  \eta(\mathfrak{a}).
$$
Then, we can easily see that 
\begin{equation}\label{partial:Lfunction:expression:equation}
L_K(s,\eta\cdot\big(\lambda\circ N)\big)=\sum_{m>0} \frac{\lambda(m)c_\eta(m)}{m^s}
\end{equation}
for any Dirichlet character $\lambda$.
Let $\omega:(\mathbb{Z}/p\mathbb{Z})^\times\rightarrow\mu_{p-1}\subset\mathbb{C}^\times$ be the Teichm{\" u}ller character.
Let $\langle\cdot\rangle$ be the canonical projection defined by
$$
\langle\cdot\rangle:(\mathbb{Z}/p^2\mathbb{Z})^\times\rightarrow \frac{1+p\mathbb{Z}}{1+p^2\mathbb{Z}},\ a\mapsto a\cdot\omega^{-1}(a).
$$
\begin{defn}
For a Hecke character $\xi$ over $K$, let us define a partial Hecke $L$-function as follows:
$$
L^{\circ}_K(s,\xi):=\sum_{\langle m \rangle\neq 1}\sum_{ N(\mathfrak{a})=m} \frac{\xi(\mathfrak{a})}{N(\mathfrak{a})^s}=\sum_{\langle m\rangle\neq 1}\frac{c_\xi(m)}{m^s},
$$
where $m$ and $\mathfrak{a}$ runs through the positive integers $m$ with $\langle m\rangle\neq 1$ and the integral ideals of $K$ with given norm $m$, respectively. 
\end{defn}
In this section, let us denote by $\psi$ a generator of the following cyclic group: 
$$\mathrm{Hom}\big((\mathbb{Z}/p^2\mathbb{Z})^\times,\mu_p\big).$$
We can express the partial Hecke $L$-function as a twisted average of Hecke $L$-functions:

\begin{prop}\label{partial:Lfunction:expression}
For any $j$,
$$
L^{\circ}_K\big(s,\eta\cdot(\omega^j\circ N)\big)=\frac{p-1}{p}L_K\big(s,\eta\cdot(\omega^j\circ N)\big)-\frac{1}{p}\sum_{i=1}^{p-1}L_K\big(s,\eta\cdot(\psi^i\omega^j\circ N)\big).
$$ 
\end{prop}
\begin{proof}
Immediate by the following orthogonality relation
\begin{align*}
\sum_{i=1}^{p-1}\psi^i(m)=\begin{cases}
			p-1 &\mbox{ if } \langle m \rangle=1 \\
			-1 &\mbox{ if } \langle m \rangle\neq 1 \\
			0 &\mbox{ otherwise, i.e., } p\mid m 
		\end{cases}
	\end{align*}
and equation (\ref{partial:Lfunction:expression:equation}).
\end{proof}

%\begin{prop}\end{prop}\begin{proof}\end{proof}

By utilizing the properties of partial $L$-functions, we obtain a result toward determination of the tame character $\eta$ by the norm average $c_\eta$: 
\begin{prop}\label{nonvanishing:averagesum:characters}
Let $A$ be a finite set of totally odd primitive ray class characters over $K$ modulo $\mathfrak{m}$.
Let $a_\eta$ be constants satisfying
$$
\sum_{\eta\in A}a_\eta c_\eta(m)=0
$$
for any positive integers $m$ such that $\langle m\rangle\neq 1$.
Let us assume that $K$ is solvable and $p$ does not divide $N(\mathfrak{m})$. Then 
$$
\sum_{\eta\in A} a_\eta \tau(\eta)\eta_{\rm f}(p^2)\text{ and }\sum_{\eta\in A} a_\eta \tau(\eta)\eta_{\rm f}(p^2q)=0
$$
for any odd rational primes $q$ does not divide $pN(\mathfrak{m})$.
\end{prop}

\begin{proof}
In this proof, let us denote by $\varpi=\omega\circ N$ and $\chi=\psi\circ N$, which are primitive cyclotomic characters over $K$ modulo $p$ and $p^2$ by Proposition \ref{cyclo}, respectively.
From Proposition \ref{partial:Lfunction:expression}, equation (\ref{partial:Lfunction:expression:equation}) and our assumption, we obtain that
\begin{equation}\label{limitwillbetaken2}
\sum_{\eta\in A}a_\eta L_K(s,\eta\varpi^j)=\frac{1}{p-1}\sum_{i=1}^{p-1}\sum_{\eta\in A}a_\eta L_K(s,\eta\chi^i\varpi^j)
\end{equation}
for any $j$.
If $p-1\nmid j$, then by the functional equation of Hecke $L$-functions,
\begin{align*}
&p^{ds} \sum_{\eta}a_\eta\tau(\eta\varpi^{j})L_K(1-s,\overline{\eta\varpi^j})
\\
&=\frac{1}{p-1}\sum_{i}\sum_{\eta}a_\eta\tau(\eta \chi^i\varpi^j)L_K(1-s,\overline{\eta \chi^i\varpi^j}).
\end{align*}
If $p-1\mid j$, then $\eta \varpi^j$ is imprimitive and $\eta$ is a primitive character which induces an imprimitive character $\eta \varpi^j$, so we obtain the following:
\begin{align}\label{limitwillbetaken}
\frac{p^{2ds}}{L_{p}(s,\eta )}\sum_\eta a_\eta\tau(\eta)L_K(1-s,\overline{\eta })=\frac{1}{p-1}\sum_{i=1}^{p-1}\sum_\eta a_\eta\tau(\eta \chi^i)L_K(1-s,\overline{\eta \chi^i}),
\end{align}
where $L_{p}(s,\eta)$ is the Euler factor of Hecke $L$ of $\eta$ dividing $p$:
$$
L_{p}(s,\eta)=\prod_{\mathfrak{p}|p}\Big(1-\frac{\eta(\mathfrak{p})}{N(\mathfrak{p})^s}\Big)^{-1}.
$$
Note that
$$
p^{2ds}L_p(s,\eta)^{-1}=o(1)\text{ and }L_K(1-s,*
)=1+o(1)\text{ as }s\rightarrow-\infty.
$$
By taking the limit $s\rightarrow-\infty$ to equation (\ref{limitwillbetaken2}), equation (\ref{limitwillbetaken}) and applying Lemma \ref{gausssum:decomposition}, we obtain that
\begin{align*}
0=\sum_{i=1}^{p-1}\sum_\eta a_\eta\tau(\eta \chi^i\varpi^j)=\sum_\eta a_\eta\tau(\eta)\eta_{\rm f}(p^2)\sum_{i=1}^{p-1}\tau(\chi^i\varpi^j)\chi^i\varpi^j(\mathfrak{m})
\end{align*}
for any $j$. Thus, by Proposition \ref{twisted:average:gausssum}, we obtain the following equation:
\begin{equation}\label{nonvanishing:averagesum:characters:equation}
\sum_{\eta\in A} a_\eta \tau(\eta)\eta_{\rm f}(p^2)=0.
\end{equation}
Let $q\nmid pN(\mathfrak{m})$ be an odd rational prime. Let $\omega_q$ be a non-trivial even primitive Dirichlet character modulo $q$. Let us denote by $\varpi_q=\omega_q\circ N$, which is an even primitive character over $K$ modulo $q$ by Proposition \ref{cyclo}. Note that 
$$
\sum_{\eta}a_\eta c_{\eta\varpi_q}(m)=\omega_q(m)\sum_{\eta}a_\eta c_{\eta}(m)=0.
$$
Thus, we can replace $A$ by $\{\eta\varpi_q\mid\eta\in A\}$ of equation 
(\ref{nonvanishing:averagesum:characters:equation}), obtaining that
$$
\tau(\varpi_q)\varpi_q(p^2\mathfrak{m})\sum_{\eta\in A}a_\eta\tau(\eta)\eta_{\rm f}(p^2q)=\sum_{\eta\in A}a_\eta\tau(\eta\varpi_q)(\eta\varpi_q)_{\rm f}(p^2)=0
$$
by Lemma \ref{gausssum:decomposition}. So we are done.
\end{proof}

From this, we can prove Lemma \ref{nonvanishing:primitive}, which plays a crucial role to show the non-vanishing (Proposition \ref{prim}) and cyclotomic field generation result (Proposition \ref{cyclotomic:generation}):
\begin{proof}[Proof of {\rm Lemma \ref{nonvanishing:primitive}}]
Assume the contrary, i.e., $c_\eta(m)=0$ for any $m$ with $\langle m\rangle\neq 1$, or equivalently, $L_K^\circ(s,\eta)=0$. Set $A=\{\eta\}$ in Proposition \ref{nonvanishing:averagesum:characters}, then we obtain $\tau(\eta)\eta_{\rm f}(p^2)=0$, which is a contradiction.
\end{proof}

Let us denote by $c_\eta^\circ$ the restriction of $c_\eta$ to the set of positive integers $m$ with $\langle m\rangle\neq 1$. 
Let us denote by $\eta_{\rm f}^\prime$ the restriction of $\eta_{\rm f}$ to the set of odd rational primes $q$ such that $q\nmid pN(\mathfrak{m})$.  
From Proposition \ref{nonvanishing:averagesum:characters}, we obtain the following field generation result:
\begin{thm}\label{nonvanishing:averagesum:characters:2}
Let us assume that $K$ is solvable and $p$ does not divide $N(\mathfrak{m})$. Then, 
$$
\mathbb{Q}(\eta_{\rm f}^\prime)\subset\mathbb{Q}(c_\eta^\circ).
$$
\end{thm}

\begin{proof}
Let us choose $\sigma\in\mathrm{Gal}\big(\mathbb{Q}(c_\eta^\circ,\eta)/\mathbb{Q}(c_\eta^\circ)\big)$. Then, by Proposition
\ref{nonvanishing:averagesum:characters}, 
$$
\tau(\eta)\eta_{\rm f}(p^2)=\tau(\eta^\sigma)\eta_{\rm f}^\sigma(p^2)\text{ and }
\tau(\eta)\eta_{\rm f}(p^2q)=\tau(\eta^\sigma)\eta^\sigma_{\rm f}(p^2q)
$$ 
for odd rational primes $q\nmid pN(\mathfrak{m})$.
From this, we conclude that $\eta_{\rm f}(q)=\eta^\sigma_{\rm f}(q)$ for such primes $q$, which implies that $\sigma$ fixes $\mathbb{Q}(\eta_{\rm f}^\prime)$. Since $\sigma$ is arbitrary, we are done.
\end{proof}

\section{Main result}
Let us recall that $\eta$ is a totally odd primitive ray class character over $K$ of modulus $\mathfrak{m}$. 
Let us recall our main question:
\begin{conjecture*} For almost all totally even primitive cyclotomic characters $\chi$ over $K$ modulo a $p$-power, does the following equality hold?
$$
F\big(L_K(0,\eta\chi)\big)=F(\eta,\chi).
$$ 
\end{conjecture*}
We prove a result toward the main question:
\begin{thm}\label{main:theorem:final}
Let us assume that $K$ is solvable and $p$ does not divide $N(\mathfrak{m})$. Then,
$$
F(\eta_{\rm f}^\prime,\chi)\subset F(c_\eta^\prime,\chi)=F\big(L_K(0,\eta\chi)\big)\subset F(\eta,\chi)
$$
for almost all totally even primitive cyclotomic characters $\chi$ over $K$ modulo a $p$-power.
\end{thm}
\begin{proof}
Since $\mathbb{Q}(c_f^\circ)\subset\mathbb{Q}(c_f^\prime)$, it is immediate from Theorem \ref{main:result} and Theorem \ref{nonvanishing:averagesum:characters:2}.
\end{proof}

\begin{remark}
Our assumption on the sign of characters $\eta$ and $\chi$ is important as $\eta\chi$ is not totally odd if and only if $L(s,\eta\chi)=0$. Note that $\eta\chi$ is totally odd if and only if either $\eta$ is totally odd and $\chi$ is totally even or $\eta$ is totally even and $\chi$ is totally odd, since $\chi$ arises from a Dirichlet character. Obviously, one can repeat this paper for totally even tame character $\eta$ and totally odd cyclotomic character $\chi$ without any obstruction, which means that the present paper covers all the possible cases on $\eta$ and $\chi$.
\end{remark}

\section*{Acknowledgements}
\thispagestyle{empty}
The first named author was supported by the Basic Science Research program and the Global-LAMP program of the National Research Foundation of Korea (NRF) grant funded by the Ministry of Education (No. RS-2023-00245291 and RS-2023-00301976). We would like to thank Peter Cho, Chan-Ho Kim, Dohyeong Kim, Henry Kim, and Yoonbok Lee for helpful comments on the joint universality and Langlands functoriality.

\end{document}